\documentclass[leqno,12pt]{amsart}
\usepackage{amscd,amsfonts,amsmath,amssymb,amstext,amsthm}
\newtheorem{theorem}{Theorem}[section]
\newtheorem{lemma}{Lemma}[section]
\newtheorem{proposition}{Proposition}[section]

\numberwithin{equation}{section}
\begin{document}
\title{The ring of normal densities, Gelfand transform and smooth BKK-type theorems}
\author{Dmitri Akhiezer and Boris Kazarnovskii}
\address {Institute for Information Transmission Problems \newline
19 B.Karetny per.,127051, Moscow, Russia,\newline
{\rm D.A.:} 
{\it akhiezer@iitp.ru},  
{\rm B.K.:} 
{\it kazbori@gmail.com}.}
\thanks{Research was carried out at the Institute for Information Transmission Problems of the Russian Academy of Sciences}

\begin{abstract} We suggest an algorithm allowing to obtain some new integral-geometric formulae from the existing formulae of Crofton type. These new formulae are applied to get smooth versions of BKK theorem. The algorithm is based on the calculations
in the ring of normal densities on a manifold.
\end{abstract}
\renewcommand{\subjclassname}
{\textup{2010} Mathematics Subject Classification}
\maketitle
\renewcommand{\thefootnote}{}
\maketitle

\section{Introduction}\label{intro}

Normal densities on a manifold $X$ of dimension $n$ are the analogue of, on the one hand, differential forms on $X$ and,
on the
other hand, affine Crofton densities on ${\mathbb R}^n$ defined by I.M.Gelfand an M.M.Smirnov in \cite{GS}.
It is shown in \cite{AK} that normal densities form a graded commutative ring ${\mathfrak n}(X) $ depending
contravariantly on $X$. It is known that the proof of any specific formula of Crofton type
is reduced to the computation of a certain Gelfand transform $\Phi \mapsto p_*q^*\Phi $. Here,
$\Phi $ is a volume density on a manifold $\Gamma $ and $q^*, p_* $ are the pull-back and
the push-forward operations coming from a double fibration $Y\  {\buildrel p\over \longleftarrow}\ X\ {\buildrel q \over 
\longrightarrow} \ \Gamma$.
We prove that the density $p_*q^*\Phi $ arising from the double fibration is normal.
It turns out that for the product of double fibrations the corresponding 
normal densities are multiplied, see Theorem \ref{gelfand product}. This result leads to an algorithm
giving new integral-geometric formulae on the basis of existing formulae
of Crofton type. In particular, we deduce the Crofton formula for the product
of spheres or Euclidean spaces from the classical formulae for one sphere or one ${\mathbb R}^n$, see Sect.\ref{double}.
Given $n$ finite dimensional vector spaces of smooth functions on $X$, one can define, in many different ways, the
average number of common zeros of $n$ functions $f_i \in V_i$, $i=1,\ldots,n$. We refer to the computation
of such a number  as a smooth version of BKK theorem \cite{Be}. One such computation is done in \cite{AK}. In Sect.\ref{BKK} we give another 
result of this type. As in \cite{AK} and \cite{Be}, the answer depends on the mixed volume of some convex bodies.

\section{Preliminaries}\label{preliminaries}
We gather here some definitions and results from \cite{AK} which will be used later on. Let $V$ be a real vector space of dimension $n$.
Throughout the paper we use the following notations:

- ${\rm Gr}_m(V)$ is the Grassmanian, whose points are vector subspaces
of dimension $m$ in $V$;

 - ${\rm D}_m(V)$ is the cone over ${\rm Gr}_m(V)$, whose points are decomposable $m$-vectors;

- ${\rm AGr}^m(V)$ is the affine Grassmanian, whose points are affine
subspaces of codimension $m$ in $V$.

\medskip
\noindent
{\it The ring of normal measures.}
A signed Borel measure $\mu $ on ${\rm AGr}^m (V)$ is called {\it normal}, if $\mu $ is translation invariant and finite on compact sets. The space of normal measures on ${\rm AGr}^m(V)$ is denoted by ${\mathfrak m}_m(V)$. 
The graded vector space ${\mathfrak m}(V) = {\mathfrak m}_0(V) \oplus {\mathfrak m}_1(V) \oplus \ldots \oplus
{\mathfrak m}_n(V)$ has a structure of a graded ring. Namely,
let 
$${\mathcal D}_{p,q} = \{(G,H) \in {\rm AGr}^p(V) \times {\rm AGr}^q(V)\ \vert \ {\rm codim}\, G\cap H \ne p+q\}.$$
Then we have the mapping 
$$P_{p,q}: {\rm AGr}^p(V) \times {\rm AGr}^q(V) \setminus {\mathcal D}_{p,q} \to {\rm AGr}^{p+q}(V),$$
given
by $P_{p,q}(G,H) = G\cap H$. For $\mu \in {\mathfrak m}_p(V), \, \nu \in {\mathfrak m}_q(V), \, p+q \le n,$
the measure $\mu \cdot \nu \in{\mathfrak m}_{p+q}(V)$
is defined by
$(\mu \cdot \nu ) (D) = (\mu \times \nu)(P_{p,q}^{-1}(D)),$
where $D \subset {\rm AGr}^{p+q}(V)$ is a bounded domain.

\medskip
\noindent
{\it The ring of normal densities on a manifold.}
Recall that an $m$-density on $V$ is a continuous function 
$\delta : {\rm D}_m (V) \to {\mathbb R}$, such that $\delta (t\xi) = \vert t \vert \delta(\xi)$ for all $t\in {\mathbb R}$.
One example of an $m$-density is the absolute value of an exterior $m$-form.
An $m$-density on a manifold $X$ is an $m$-density $\delta _x$ on each tangent space $T_x$,
such that the assignment $x\mapsto \delta _x$ is continuous.
We consider a density on $V$ as a translation invariant density on the corresponding affine space.
Such a density is called {\it normal} if it is obtained from a normal measure
on ${\rm AGr}^m(V)$ by the construction which we now describe.  Namely, for $D \subset V$ put
$${\mathcal J}_{m,D} = \{H \in {\rm AGr}^m(V) \ \vert \ H \cap D \ne \emptyset \}$$
and  for $\xi_1,\ldots,\xi_m \in V$ denote by $\Pi_\xi \subset V$ the parallelotope generated by $\xi _i$.
Given $\mu \in {\mathfrak m}_m$,  define a normal $m$-density $\chi _m(\mu)$ by 
$$\chi _m(\mu)(\xi_1\wedge \ldots \wedge \xi_m ) = \mu ({\mathcal J}_{m,\Pi_\xi}).$$
Then we have the spaces of normal $m$-densities
${\mathfrak n}_m = {\mathfrak n}_m(V) , m = 0,1,\ldots,n$, their direct sum
${\mathfrak n}(V) =\oplus \, {\mathfrak n}_m(V)$ and the linear map
$\chi = \oplus \chi_m : {\mathfrak m}(V) \to {\mathfrak n}(V)$. It is easily seen that ${\rm Ker}\, \chi$
is a homogeneous ideal of the graded ring ${\mathfrak m}(V)$. Therefore ${\mathfrak n}(V)$ carries a structure of a graded ring, the {\it ring of normal densities on $V$}.

\medskip
\noindent
{\bf Remark.} {\it See \cite{AK}, Sect.\,3.3., for the connection of ${\mathfrak n}(V)$ with the ring of valuations
on convex bodies defined by S.\,Alesker \cite{Al}}.

\medskip
\noindent
 Let $X$ be a smooth manifold of dimension $n$. A normal $m$-density on $X$
is a continuous function $x \mapsto \delta _x \in {\mathfrak n}_m(T_xX)$.
The pointwise construction of product leads to the definition of {\it the ring of normal densities on $X$},
denoted by ${\mathfrak n}(X)$.
As an example of a normal density, we will prove in the next section
that the absolute value $\vert \omega \vert $ of a decomposable differential $m$-form $\omega $
is a normal $m$-density. Morever, one can show that for decomposable forms multiplication
of forms agrees with multiplicaion of densities. More precisely, if $\omega_1, \omega _2$ are decomposable forms then $\vert \omega _1 \wedge \omega  _2\vert = \vert \omega _1 \vert \cdot \vert \omega _2\vert$.

\medskip
\noindent
{\it Contravariance of ${\mathfrak n}(X)$}.
The assignment $X \to {\mathfrak n}(X)$ is a contravariant functor from the category of smooth manifolds to the
category of commutative graded rings. Namely, the pull-back of a normal density is normal
and the product of pull-backs is the pull-back of the product of normal densities. This property is based on the following
construction of {\it the pull-back of a normal measure}.  

For a linear map $\varphi: U \to V$ of vector spaces one can define the pull-back operation
$${\mathfrak m}_m (V) \ni \mu \mapsto \varphi^*(\mu)\in {\mathfrak m}_m(U),$$
so that $\chi _m (\varphi^*\mu ) = \varphi^*(\chi _m (\mu ))$.
For $\mu \in {\mathfrak m}_p(V), \nu \in {\mathfrak m}_q(V)$ one has
$$\chi _p(\varphi ^*\mu)\cdot \chi_q(\varphi ^*\nu) = \chi_{p+q}(\varphi ^*(\mu\cdot \nu)).$$
We need the notion of pull-back only for epimorphisms. If $\varphi $ is surjective then we have the closed 
imbedding 
$\varphi _*:{\rm AGr}^m(V) \to {\rm AGr}^m(U)$ defined by taking the preimage under $\varphi $ of an affine subspace in $V$. By definition, the measure $\varphi ^*\mu$ is
supported on $\varphi _*{\rm AGr}^m(V)$ and 
$$
\varphi^*\mu(\Omega) = \mu (\varphi _*^{-1}(\Omega ))$$ for any Borel set $\Omega \subset \varphi_*{\rm AGr}^m(V)$.
In other words,
$$\int_{\varphi _*{\rm AGr}^m(V)}f\cdot d(\varphi ^*\mu) = \int _{{\rm AGr}^m(V)} \tilde f \cdot d\mu$$
for any continuous function $f$ with compact support on $\varphi _*{\rm AGr}^m(V)$ and $\tilde f(H) = f(\varphi ^{-1}H)$,
$H \in {\rm AGr}^m(V)$.

\medskip
\noindent
{\it Normal 1-densities.} It is shown in \cite{AK} that any smooth 1-density is normal.
This follows from certain properties of the cos-transform, see, e.g., \cite{AB}. Therefore ${\mathfrak n}(X)$ contains
the ring generated by smooth 1-densities.

\medskip
\noindent
{\it Normal densities and convex bodies.} Let $A_1\ldots, A_m$ be convex bodies in the dual space $V^*$. The $m$-density
$d_m(A_1, \ldots, A_m)$ on $V$ is defined as follows. Let $H$ be the subspace of $V$ generated by $\xi _1, \ldots, \xi_m\in V$,
$H^\bot \subset V^*$  the orthogonal complement to $H$, and $\pi _H: V^* \to V^*/H^\bot $
the projection map.
Consider $\xi_1\wedge \ldots \wedge \xi_m$ as a volume form on $V^*/H^\bot $
and define $d_m(A_1, \ldots, A_m)(\xi _1, \ldots, \xi _m)$ as the mixed
$m$-dimensonal volume of $\pi_HA_1, \ldots, \pi_HA_m$.
For example, $d_1(A)(\xi) = h(\xi) - h(-\xi )$, where $h : V \to { R}$ is the support
function of a convex body $A$. If $A$ is a smooth convex body of full dimension then the density $d_1(A)$
is smooth and, consequently, normal. Using the pull-back operaion
for normal measures and densities, one can show the same if $A$ 
is a smooth body of smaller dimension. 

Suppose $A_1, \ldots, A_m$ are centrally symmetric smooth convex bodies. Then Theorem 7 in \cite{AK}
asserts that
\begin{equation}
d_1(A_1)\cdot \ldots \cdot d_1(A_m) = m!\,d_m(A_1, \ldots, A_m). \label{*}
\end{equation}
Thus all densities $d_m(A_1, \ldots, A_m)$ are normal.

\medskip
\noindent
{\it Finsler convex sets.} Let $X$ be a smooth manifold of dimension $n$. Suppose that for every $x \in X$ we are given
a convex body ${\mathcal E}(x) \subset T_x^*$ depending continuously on $x\in X$.
The collection ${\mathcal E} = \{{\mathcal E}(x) \, \vert \, x \in X\}$ is called a Finsler convex set in $X$.
For
$m$ Finsler convex sets ${\mathcal E}_1, \ldots, {\mathcal E}_m$ denote by
$D_m({\mathcal E}_1, \ldots, {\mathcal E}_m)\in {\mathfrak n}_m(X)$ the $m$-density whose value at $x\in X$
equals $d_m({\mathcal E}_1(x), \ldots, {\mathcal E}_m(x))$. 
By definition, $D_m({\mathcal E}) = D_m({\mathcal E}_1, \ldots, {\mathcal E}_m)$,
where ${\mathcal E}_i = {\mathcal E}, i=1,\ldots,m$.
The integral $\int _XD_n({\mathcal E})$ is the symplectic volume of the domain 
$$\bigcup _{x\in X}{\mathcal E }(x) \subset T^*(X)$$
with respect to the standard symplectic structure on the cotangent bundle.
The integral
$\int _XD_n({\mathcal E}_1, \ldots, {\mathcal E}_n)$ is called the mixed (symplectic) volume
of ${\mathcal E}_1, \ldots, {\mathcal E}_n$.
Under the assumptions of (\ref{*}) on $A_i = {\mathcal E}_i(x)$ at every $x \in X$ we have
\begin{equation}
D_1({\mathcal E}_1)\cdot \ldots \cdot D_1({\mathcal E}_m) = m!\, D_m({\mathcal E}_1, \ldots, {\mathcal E}_m). \label{**}
\end{equation}
Given a $C^\infty $ map $f : X \to Y$ and a Finsler convex set ${\mathcal E}$ in $Y$, one can define its inverse image,
a Finsler convex set in $X$,  
by $f^*{\mathcal E}(x) = d^*f({\mathcal E}(f(x))$, where $d^*f$
is the dual to the differential of $f$. It is easy to see that 
\begin{equation}
D_m(f^*{\mathcal E}_1, \ldots, f^*{\mathcal E}_m) = f^*(D_m({\mathcal E}_1, \ldots, {\mathcal E}_m)), \label{***}
\end{equation}
i.e., the operation of inverse image of Finsler convex sets commutes with the pull-back of densities $D_m$.

\section{Decomposable forms and normal densities}\label{decomp}

\bigskip

Let $V$  be a real vector space of dimension $n$ and $V^*$ the dual space.
Recall that an exterior form $\omega \in \bigwedge^m V^*$ is said to be {\it decomposable} if $\omega = \varphi _1 \wedge  \ldots  
\wedge \varphi_m$, where $\varphi _i$ are 1-forms. 
For a decomposable form $\omega \ne 0$ the subspace in $V^*$ generated by $\varphi _1, \ldots, \varphi_m$
has dimension $m$. 
Its orthogonal complement in $V$ is called the kernel of $\omega $ and denoted by
${\rm Ker}\, \omega$. Of course, ${\rm Ker}\, \omega $ is given by the equations $\varphi _1 = \ldots =\varphi _m = 0 $.
A volume form is a non-zero form of maximal degree. A volume form on a given vector space
is unique up to a scalar multiple and decomposable.

\begin{lemma}\label {decomp1} An $m$-form $\omega $ is decomposable if and only if $\omega $ is the pull-back of a volume form $\omega^\prime $ under an epimorphism $\pi : V \to V^\prime$.

\end{lemma}

\begin{proof} Let $\omega $ be a decomposable $m$-form, $\omega = \varphi _1 \wedge \ldots \wedge \varphi _m$, 
and $V^\prime =V/{\rm Ker\, \omega} $. Denote by $\pi : V \to  V^\prime $ the canonical map. Then each $\varphi _i$ is the pull-back of
some $\varphi _i^\prime $, where $\varphi_1^\prime , \ldots,  
\varphi_m  ^\prime$ are linearly independent 1-forms on $V^\prime $. The wedge product $\omega ^\prime = \varphi _1^\prime \wedge
\ldots \wedge \varphi_m^\prime$ is a volume form on $V^\prime $ and its pull-back $\pi ^*\omega ^\prime $ is the initial form $\omega $.
The converse assertion is obvious.
\end{proof}

For any $\omega \in \bigwedge ^m V^*$ and $\xi = (\xi _1, \ldots, \xi _k ),\ \xi_i \in V,\ k\le m$, the inner product 
$\xi \lrcorner \, \omega \in \bigwedge ^{m-k} V^*$ is defined by
$$(\xi\lrcorner \, \omega) (v_1, \ldots, v_{m-k}) = \omega(\xi_1, \ldots, \xi_k, v_1, \ldots ,v_{m-k}).$$

\begin{lemma}\label{decomp2} If $\omega \ne 0$ is decomposable then $\xi \lrcorner \, \omega $ is also 
decomposable. If $\xi _1, \ldots, \xi _k $ are linearly independent modulo ${\rm Ker}\, \omega$
then 
$${\rm Ker}\, \xi \lrcorner \, \omega = {\rm Ker}\, \omega +{\mathbb R}\, \xi _1 + \ldots + {\mathbb R}\, \xi_k. \eqno{(*)}$$
Otherwise $\xi\lrcorner\, \omega = 0 $.
\end{lemma}

\begin{proof} Let $\omega = \pi^*\omega ^\prime$ be the pull-back of $\omega ^\prime \in \bigwedge ^m V^\prime$ under an epimorphism $\pi: V \to V^\prime$.
Write $\pi (\xi)$ for $(\pi (\xi_1),\ldots,\pi(\xi_k))$. 
Then $\xi\lrcorner \,\omega = \pi^*(\pi(\xi)\lrcorner\, \omega ^\prime)$ and 
${\rm Ker}\, \xi \lrcorner \, \omega  = \pi^{-1} \bigl ({\rm Ker}\, \pi(\xi) \lrcorner \, \omega ^\prime \bigr )$.
By Lemma \ref{decomp1} we may assume that  $\omega ^\prime $ is a volume form on $V^\prime$. 
If $\pi(\xi_1), \ldots, \pi (\xi_k)$ are
linearly independent then $\pi (\xi)\lrcorner \,\omega ^\prime$ is the pull-back of a volume form on $V^\prime /U $, where
$U = {\mathbb R}\pi (\xi_1) +\ldots + {\mathbb R}\pi (\xi_k).$  Thus $\xi \lrcorner\, \omega$ is decomposable with kernel given by $(*)$. Otherwise 
$\pi (\xi) \lrcorner \, \omega ^\prime= 0 $ and $\xi \lrcorner \, \omega = 0 $.
\end{proof}

Suppose $\omega \ne 0$ is a decomposable $m$-form on $V$, where $1\le m \le n $. We associate with $\omega  $ a
normal measure
$\mu_\omega \in {\mathfrak m}_m(V)$. Namely, by Lemma \ref{decomp1}  we have $\omega = \pi^*\omega ^\prime $, where
$\pi : V \to V^\prime = V/{\rm Ker}\, \omega$  is the canonical quotient map and
$\omega ^\prime $ is a volume form on $ V^\prime$.   Let $\mu ^\prime $ be the corresponding Lebesgue measure on $V^\prime $, so that
$\chi _m(\mu ^\prime ) = \vert \omega ^\prime \vert$.
Applying
the pull-back operation for normal measures, we put
 $\mu  _\omega= \pi ^*\mu ^\prime \in {\mathfrak m}_m(V)$. 
Then
$$\chi _m (\mu _\omega) = \chi _m (\pi ^*\mu ^\prime)  = \pi^*(\chi _m \mu^\prime) = \pi^*(\vert\omega ^\prime \vert) = 
\vert \omega \vert ,$$  
showing that $\vert \omega \vert$ is a normal $m$-density.

For $\omega = 0$ we put $\mu _0 = 0$.
\begin{lemma} \label{cont}The map
$${\rm D}_{m} (V^*) \ni \omega \mapsto \mu_\omega \in {\mathfrak m}_m$$
is continuous.
\end{lemma}
\begin{proof} Let $\omega _i \to \omega $ as $i \to \infty $. We want to prove that $\mu _i=\mu _{\omega _i} \to \mu =\mu_\omega $.
Put $L_i = {\rm Ker}\, \omega _i$. By compactness of ${\rm Gr}_{n-m}(V)$ we may assume that $L_i \to L$,
where $L$ is some vector subspace of codimension $m$ in $V$. Note that if $\omega \ne 0$ then $L = {\rm Ker}\, \omega $. 
Take any subspace $M \subset V$, such that $M \oplus L = V$. Then $M\oplus L_i = V$ for large $i$, so that we
have the projection maps $\pi $ and $ \pi _i$ of $V$ to $M$ parallel to $L$ and, respectively, $L_i$.
Consider the restrictions of $\omega $ and $\omega _i$ onto $M$ and
denote by $\nu  $ and $\nu _i $ the corresponding measures
on $M$.
Take any continuous function $f$ with compact support on ${\rm AGr}^m(V)$ and define the functions
$\tilde f_i, \tilde f$ on $M$ by
$$ \tilde f_i(m) = f(m+L_i), \tilde f(m) = f(m+L), \ m\in M.$$
Observe that $\tilde f_i$ and $\tilde f$ have compact supports bounded altogether and $\tilde f_i \to \tilde f$.
Since $\omega _ i \vert _M\to \omega \vert _M$, we have
 $\nu _i \to \nu $. Therefore
$$\int _{{\rm AGr}^m (V)}  f\cdot d\mu_i = \int _M \tilde f_i \cdot d\nu_i \to
\int _M \tilde f \cdot d\nu = \int _{{\rm AGr}^m (V)}  f\cdot d\mu $$
by the definiton of pull-back for normal measures applied to 
$\mu = \pi^*\nu $ and $\mu _i= \pi^*_i \nu_i$.
\end{proof}

We now recall the definition of {\it push-forward} for densities, see \cite{AK},\cite{PF},\ and \cite{GS}. Let
$X,Y$ be smooth manifolds, $X$ a locally trivial fiber bundle over $Y$ with compact fiber,
$p: X \to Y$ the projection map, and $k$ the dimension of the fiber $F_y = p^{-1}(y),\ y \in Y$. Let $\delta $ be an $m$-density on $X$
and assume $k = {\rm dim}\,F_y \le m$. Then the push-forward $p_*\delta$ is an $(m-k)$-density  on $Y$,
defined as follows. Given $\eta _1, \ldots, \eta_{m-k} \in T_y(Y)$ and $x \in F_y$, choose arbitrarily $\xi_1, \ldots, \xi_{m-k} \in T_x(X)$,
so that $dp_x(\xi_i) = \eta _i$. Put $\xi = \xi_1 \wedge \ldots \wedge \xi_{m-k}$ and consider the 
density $\xi \lrcorner \,\delta $ on the fiber, defined by 
$$\xi\lrcorner \, \delta (f_1\wedge \ldots \wedge f_k) = \delta(\xi\wedge f_1 \ldots \wedge f_k),$$
where $f_1, \ldots, f_k$ are tangent to the fiber $F_y$ at $x$. Here, $\xi\lrcorner \, \delta (f_1\wedge \ldots \wedge f_k)$
does not depend on the choice of $\xi _i$. By definition,
$$p_*\delta (\eta _1 \wedge \ldots \wedge \eta_{m-k}) = \int _{F_{y}} \xi\lrcorner \, \delta.$$

An exterior $m$-form $\omega $ on $X$ is said to be {\it decomposable}  
if $\omega $ is the product of 1-forms on $X$. If $\omega $ is the product of 1-forms
in a neighborhood of every point then $\omega $ is called {\it locally decomposable}. For a locally decomposable $m$-form $\omega $ its absolute value $\delta = \vert \omega \vert$ is a normal $m$-density. We will say that a normal $m$-density $\delta $
{\it locally corresponds to a decomposable form} if
for every point $x \in X$
there is a neighborhood of $x$ in which $\delta $ is the absolute value of a decomposable $m$-form defined in that neighborhood. Note that a top-order density is normal and locally corresponds to a decomposable form. 

\begin{theorem}\label{normal}
Let $\delta $ be a normal $m$-density on $X$, which locally corresponds to a decomposable form, and let
$p : X \to Y$ be the projection map of a fiber bundle whose fiber $F_y, \, y \in Y,$ is a compact manifold
of dimension $k \le m$. Then the push-forward $p _*\delta $ is a normal $(m-k)$-density on $Y$. 
\end{theorem}
\begin{proof}
Fix a point $x \in X$ and write $\delta = \vert \omega \vert$ in a neighborhood of $x$,
where $\omega $ is a decomposable $m$-form in that neighborhood. Put $V = T_x(X), W = T_y(Y)$, where $y = p(x)$, and $L = T_x(F_y)$.
Choose a complementary vector subspace $S$ to $L$ in $V$, so that $dp_x\vert _S $ is an isomorphism
 $S \buildrel \sim \over \to W$. By Lemma \ref{decomp2} $f\lrcorner\, \omega $
 is a decomposable $(m-k)$-form on $V$ for any $f=(f_1, \ldots, f_k)$. Restrict this form to $S$ and then carry it down to $W$
using the isomorphism $S\simeq W$. For $f_1, \ldots, f_k \in L$ the resulting 
form does not depend on the choice of $S$. We call this form $\omega _x$. By construction, $\omega _x$
is a decomposable $(m-k)$-form on $W$ and
$$\omega _x (\eta _1, \ldots, \eta _{m-k}) = (f \lrcorner \, \omega )(\xi_1, \ldots, \xi_{m-k}),$$
where
$\xi_1, \ldots, \xi_{m-k} \in S$ are uniquely determined by $dp_x(\xi_i) =\eta_i$.
Let $\mu _x = \mu_{\omega _x} \in {\mathfrak m}_{m-k}(W)$ be the distinguished normal measure 
associated with $\omega _x$, such that $\chi_ {m-k}(\mu _x )= \vert\omega _x\vert = \delta _x$. When $x$ varies along the fiber $F_y$, the family
of forms $\{\omega _x\}$ is continuous. By Lemma \ref{cont} the family of measures $\{{\mu}_x\} $ is also
 continuous. Since $\mu _x$ and $\delta _x$ depend on $f_1\wedge \ldots \wedge f_k \in \bigwedge^kT_x(F_y)$ as densities, we can define a normal measure by the 
integral 
$
\mu = \int _{F_y}\mu_x$. Then
$$\chi _{m-k}(\mu ) (\eta _1 \wedge \ldots \wedge \eta _{m-k}) = \int _{F_y}\chi_{m-k}(\mu_x)(\eta _1\wedge \ldots
\wedge  \eta _{m-k})=$$
$$= \int _{F_y}\delta_x (\xi_1, \ldots, \xi_{m-k})=
\int _{F_y} \xi\lrcorner \, \delta _x . 
$$
By the definition of push-forward this equals $p_*\delta (\eta _1, \ldots, \eta _{m-k}).$
\end{proof}

\noindent
{\bf Remark.}
{\it In the above notation, let $\omega _x \ne 0$ and  
and let $f_1,\ldots,f_k$ be any fixed basis of $T_x(F_y)$.
Then 
$\xi \lrcorner\, \omega _x (f_1, \ldots, f_k) = 0$ for all sets of vectors $\eta _1, \ldots, \eta _{m-k}$ if and only if 
$T_x(F_y) \cap {\rm Ker}\,\omega_x  \ne \{0\}$.}

\begin{proof} If  
$T_x(F_y) \cap {\rm Ker}\,\omega _x  \ne \{0\}$ then $f_1, \ldots , f_k$ are linearly dependent
modulo ${\rm Ker}\, \omega _x$ and $\xi \lrcorner\, \omega _x $ restricted to $T_x F_y$ is zero for all $\xi $.
If $T_x(F_y) \cap {\rm Ker}\,\omega _x  = \{0\}$ then 
${\rm codim}\, (T_x(F_y) + {\rm Ker}\, \omega _x) = m-k $ and we can find $\eta _1, \ldots, \eta_{m-k}$
whose liftings $\xi _1, \ldots, \xi_{m-k}$ are linearly independent modulo
$T_x(F_y) + {\rm Ker\, }\omega _x$. Then $\xi\lrcorner \, \omega _x (f_1, \ldots, f_k) \ne 0$.
\end{proof}

\noindent
This shows that the direct image $p_*\delta $ vanishes at $y$ if and only if the kernel of $\omega$
has non-zero intersection with $T_xF_y$ everywhere on $F_y$.

\section{The product theorem}

The product of normal measures $\mu $ on ${\rm AGr}^p(V)$ and $\nu $ on ${\rm AGr}^q(V)$ is a normal
measure $\mu \cdot \nu $ on ${\rm AGr}^{p+q}(V)$, see Section \ref{preliminaries}. 
We will apply this definition to $V= V_1 \oplus V_2$ and $\mu$, $\nu$
coming from $V_1$, $V_2$, respectively. More precisely, we 
consider the embeddings
$$\iota _1:{\rm AGr}^p(V_1) \to {\rm AGr}^p(V),\ \ G\mapsto G\oplus V_2,$$
$$\iota _2:{\rm AGr}^q(V_2) \to {\rm AGr}^q(V),\ \  H \mapsto V_1 \oplus H,$$
and identify ${\rm AGr}^p(V_1), {\rm AGr}^q(V_2)$ with their images under $\iota _1, \iota _2$.
For $\mu \in {\mathfrak m_p}(V_1)$ and $\nu \in {\mathfrak m_q}(V_2)$ we put 
$$\mu \cdot \nu = \iota _{1,*}(\mu )
\cdot \iota _{2,*}(\nu ) \in {\mathfrak m}_{p+q}(V).$$
Since $({\rm AGr}^p(V_1)\times{\rm AGr}^q(V_2))\cap{\mathcal D}_{p,q} = \emptyset $
and $P_{p,q} \cdot (\iota _1 \times \iota _2)$ is an embedding, we have $\mu \cdot \nu = \mu \times \nu $.

In what follows, we consider the product of two manifolds $X=X_1 \times X_2$
and use the same notation for differential forms on $X_1, X_2$ and their liftings to $ X$.
The same convention holds for densities.
Let $\delta _1 \in {\mathfrak n}_{m_1}(X_1), \delta _2\in {\mathfrak n}_{m_2}(X_2)$ be normal densities on $X_1, X_2$ which locally correspond to 
decomposable forms, $m= m_1 + m_2$, and
$\delta = \delta _1 \delta _2 \in {\mathfrak n}_m(X)$. Suppose $p_1:X_1 \to Y_1,\ p_2 : X_2 \to Y_2$ 
are two fiberings whose fibers are compact and have dimensions $k_1\le m_1$, $k_2 \le m_2$.
By Theorem \ref{normal} the push-forwards of $\delta _1, \delta _2$ and $\delta $
are normal densities.

\begin{theorem}\label{product}

One has $$(p_1\times p_2)_* (\delta) = p_{1*}(\delta _1)p_{2*}(\delta _2)$$
in the ring ${\mathfrak n}(Y_1\times Y_2)$.
\end{theorem}

\begin{proof} Let $x_1 \in X_1, x_2 \in X_2$. As in the proof of 
Theorem \ref{normal}, construct the normal measures $\mu _{x_1}$ and $\mu _{x_2}$, such that
$\chi_{m_i-k_i}(\mu_{x_i}) = \delta _{x_i}, \ i=1,2.$ Then $\chi _{m-k}(\mu_{x_1}\cdot\mu_{x_2}) = \delta _{x_1}\cdot \delta
_{x_2} = \delta _x$. Now, $\mu _{x_1}$ and $\mu _{x_2}$ come from two different direct summands
of the decomposition $T_x(X) = T_{x_1}(X_1) \oplus T_{x_2}(X_2)$. Therefore, using the above observation, we
can replace the product of measures by the direct product, i.e.,
$\chi_{m -k}(\mu_{x_1}\times \mu _{x_2}) = \delta _x$.
Recall that $\mu _{x_1}$ and $\mu _{x_2}$ behave
as densities when $x_1$ and $x_2$ vary along the fibers $F_{y_1}=p_1^{-1}(y_1)$ and $F_{y_2} = 
p_1^{-1}(y_2)$, where $y_i 
=p_i(x_i)$.
We have
$$ \int _{F_{y_1}\times F_{y_2}} \mu_{x_1}\times\mu_{x_2}= \int _{F_{y_1}} \mu_{x_1} \times \int_{F_{y_2}}\mu_{x_2}.
$$
Take two decomposable covectors
$\eta _1 \in \bigwedge^{m_1-k_1}(T_{y_1}(Y_1))$ and
$\eta _2 \in \bigwedge^{m_2-k_2}(T_{y_2}(Y_2))$, apply $\chi $
and evaluate the both sides of the equality on $\eta_1\wedge\eta_2$ as in the proof of Theorem \ref{normal}. This yields the result.
\end{proof}

\section{Double fibrations}\label{double}
Our goal are certain Crofton formulae, and we start by recalling the approach using double fibrations \cite{PF}, \cite{GS}.  
A double fibration is a diagram  
$$\begin {matrix}
&&\ X && \cr
&&p \swarrow\ \searrow q && \cr
&&\ Y \hspace{30 pt} \Gamma \,,&& \cr 
\end{matrix}$$
in which $X,Y$, and $\Gamma $ are smooth manifolds, $p: X\to Y$ and $q: X \to \Gamma $ are projection maps of locally trivial fiber bundles.
The fiber of $p$ is supposed to be compact. Let $k$ be its dimension and let $\Phi$ be an $m$-density on $\Gamma $, where $m \ge k$.
Then the $(m-k)$-density $p_*q^*\Phi $ on $Y$ is called the Gelfand transform of $\Phi $.

Suppose 
that $p\times q: X \to Y\times \Gamma$ is a smooth embedding and 
$$Y_\gamma :=p(q^{-1}(\gamma )) \subset Y \ \ {\rm and}\ \ 
\Gamma _y : =q(p^{-1}(y)) \subset \Gamma $$ are smooth submanifolds for all $\gamma \in \Gamma, y \in Y$.
If $\Phi $ is a top-order density, i.e., if $m = {\rm dim}\, \Gamma $, then the Gelfand transform
is especially important. Namely, assume that all submanifolds $Y_\gamma $ have codimension $n=m-k$ and
let $M \subset Y$ be a submanifold of dimension $n$. Then, if the integral 
$\int _{\gamma \in \Gamma} {\#}(M\cap Y_\gamma) \Phi $ is finite, it gives the average number 
of intersection points of $M$ and $Y_\gamma $ with respect to $\Phi$. 
On the other hand, the above integral
equals $\int _M p_*q^*\Phi $, see \cite {PF}.  
\begin{theorem}\label{gelfand product}
Let 
$$Y_i \ {\buildrel p_i\over \longleftarrow} \ X_i {\buildrel q_i\over \longrightarrow} \ \ \Gamma_i,\ i=1,2,\ldots,n$$
be double fibrations, $\Phi _i$ top-order densities on $\Gamma _i$, and $\Omega _i $ their Gelfand transforms on $Y_i$.
Consider the product double fibration
$$Y=Y_1\times\ldots\times Y_n \ {\buildrel p\over \longleftarrow }\ X=X_1\times \ldots \times X_n \ {\buildrel q 
\over \longrightarrow} \ \Gamma = \Gamma_1\times \ldots \times \Gamma_n,$$
where $p=p_1\times \ldots \times p_n,\ q=q_1\times \ldots \times q_n$.
Let $\Phi = \Phi _1 \cdot \ldots \cdot \Phi _n$ be the corresponding top-order density on $\Gamma $ and
$\Omega $ its Gelfand transform on $Y$.
Then 
$\Omega $ and $\Omega_i,\ i=1,\ldots,n,$ are normal densities and 
$$\Omega = \Omega _1 \cdot \ldots \cdot \Omega _n$$
in the ring ${\mathfrak n}(Y)$.

\end{theorem}
\begin{proof}
This is a direct consequence of Theorem \ref{product}.
\end{proof}

\noindent
We will give two examples when the Gelfand transform can be computed explicitly. 

1)\,{\it The product of spheres.} The one sphere case is classical, see \cite{Sa}. The case of several spheres was considered
in \cite{AK}, but we want to give another proof. 
For $i = 1,\ldots, n$ let $E_i$ be an Euclidean space with scalar product $\langle .,\rangle_i$, $S_i \subset E_i$ the unit sphere,
$Y _i= \Gamma _i=S_i$, and 
$$ X_i= \{(y_i,\gamma_i) \in S_i\times S_i \ \vert \  \langle y_i, \gamma _i\rangle _i = 0\},\ \ i= 1,\ldots,n.$$
Then we have the projection maps
$p_i: X_i \to Y_i, \ p_i(y_i,\gamma _i) = y_i,$
and 
$q _i: X_i \to \Gamma_i, \ q_i(y_i,\gamma_i) = \gamma _i,$
defining the double fibrations. The tangent space to $S_i$ at $y_i \in S_i$ is identified
with the orthogonal complement to ${\mathbb R}\cdot y_i$ in $E_i$.
For a tangent vector $\xi = (\xi_1, \ldots, \xi_n)$ to $Y = S_1\times \cdots\times S_n$
define ${\rm vol}_{1,i}(\xi ) = \sqrt{\langle \xi _i, \xi_i\rangle}$. This is a 1-density on $Y$.
Note that these 1-densities are normal, so they can be multiplied.
Take the volume density $\Phi _i$ on $S_i$,
lift it to the product $\Gamma = S_1 \times \ldots \times S_n$, and denote the lift again by $\Phi _i$. Define $p, q$ and
$\Phi $ as in the preceding theorem.
\begin{theorem}[\rm \cite {AK}, Theorem 8]
The Gelfand transform $\Omega $ of the volume density $\Phi $ is given by
$$\Omega = {1\over \pi ^n}\cdot {\rm vol} _{1,1} \cdot \ldots \cdot {\rm vol}_{1,n}.$$
\end{theorem}
\begin{proof}
Let $\delta _i = q_i^*(\Phi _i)$ and $\delta = q^*(\Phi)$.
Then  $\delta = \delta_1 \cdot \ldots \cdot \delta_n$ and
$$p_*q^*(\Phi ) = p_*(\delta ) = p_{1*}(\delta _1)\cdot \ldots \cdot p_{n*}(\delta _n) $$
by Theorem \ref{gelfand product}.
Finally, 
$$p_{i*}(\delta _i)
={1\over \pi}{\rm vol}_{1,i}, \ \ \ i=1,\ldots,n,$$
by the Crofton formula for curves on the sphere.
\end{proof}

2)\,{\it The product of Euclidean spaces.} For a finite number of vector spaces $V_1, \ldots, V_n$
put $Y = V_1 \times \ldots \times V_n$, $\Gamma = {\rm AGr}^1(V_1)\times \ldots \times {\rm AGr}^1(V_n)$
and
$$X = X_1 \times \ldots \times X_n,\ \ X_i = \{(v_i, L_i) \in V_i \times {\rm AGr}^1(V_i) \vert \  v_i \in L_i\}.$$
We have the projection maps
$p_i:X_i\to Y_i,\ q_i:X_i \to \Gamma _i,\ i=1,\ldots,n,$
defining $n$ double fibrations. As in Theorem \ref{gelfand product} put
$$ p=p_1\times \ldots \times p_n : X \to Y,\ \ q =q_1 \times \ldots \times q_n: X \to \Gamma.$$ 
Now, introduce a Euclidean metric on each $V_i$. 
Let ${\rm vol}_{1,i}(v_i)$ be the length of $v_i\in V_i$, a 1-density on $V_i$. 
Take the volume density $\Phi _i$ on ${\rm AGr}^1(V_i)$, invariant under rotations and
translations and normalized by the condition that the volume of the set
of all hyperplanes intersecting a segment of length 1 is equal to 1. Define the volume density
on $\Gamma $ by $\Phi = \Phi _1 \cdot \ldots \cdot \Phi_n$ and let $\Omega =  p_*q^*\Phi$ be the Gelfand transform of $\Phi$ .

\begin{theorem} \label{gelfand-euclid} The Gelfand transform $\Omega $ of the volume density $\Phi $
is given by
$$\Omega = {\rm vol}_{1,1}\cdot \ldots \cdot {\rm vol}_{1,n}.$$

\end{theorem}
\begin{proof}
The proof is the same as the proof of the preceding theorem. Apart from Theorem \ref{gelfand product},
it uses the classical Crofton formula for curves in a Euclidean space, which in this case yields
$p_{i *}q_i^*(\Phi _i) = {\rm vol}_{1,i}$ for all $i = 1,\ldots, n$.
\end{proof} 

\section {Two versions of BKK theorem for smooth functions} \label{BKK}

Let $X$ be a manifold of dimension $n$ and let
$V_1, \ldots, V_n$ be finite dimensional vector spaces of real $C^\infty $ functions
on $X$. Assume that each $V_i$ is equipped with a scalar product. In \cite{AK} we defined the average number of solutions of the system $f_1 = \ldots = f_n = 0$,
where $f_i \in V_i$, and proved that this number is equal to the mixed symplectic volume of certain
Finsler ellipsoids in $X$. A similar result for $X={\mathbb R}^n$
is obtained in \cite{ZK}. Our proof is based on the Crofton formula for the product of spheres
and the identity (\ref{**}).
Here, we will prove another version of BKK theorem based on the Crofton formula
for the product of Euclidean spaces. Namely, we consider 
the system of equations of the form
\begin{equation}
f_1 - a_1 = f_2 - a_2 = \ldots = f_n - a_n = 0, \label{syst}
\end{equation}
where $f_i \in V_i,\, f_i \ne 0,\,a_i \in {\mathbb R}$.
To define the average number of solutions over all such systems, denote by $V_i^*$ the dual Euclidean space to $V_i$
and regard $f_i$ as a linear functional $f_i : V_i^* \to {\mathbb R}$. System (\ref{syst}) can be viewed as a set of hyperplanes 
$H_i\subset V_i^*$ having equations $f_i = a_i$. The number of solutions is denoted by $N(H_1, \ldots, H_n)$.
The average number of solutions is the integral of $N(H_1,\ldots, H_n)$ over the product of
affine Grassmanians ${\rm AGr}^1(V_i^*)$ with respect to the invariant measure defined in Section \ref{double}.

Let $\theta _i : X \to V_i^*$ be the map assigning to any $x \in X$
the functional
$$\theta _i(x)(f) = f(x),\ \ f\in V_i.$$
Define a Finsler convex set ${\mathcal R}_i$ in $V_i^*$ by the condition that ${\mathcal R}_i(u)$ is the unit ball in $V_i$ for all $u \in V_i^*$, so that 
$$\bigcup _{u\in V_i^*}\ {\mathcal R}_i(u) \subset T^*V_i^* = V_i^*\times V_i$$ is the product of 
$V_i^*$ and the unit ball in $V_i$. Let ${\mathcal E}_i$ be the inverse image $\theta _i^*{\mathcal R}_i$.
Then ${\mathcal E}_i(x)$ is an ellipsoid in $T^*_xX$, possibly of smaller dimension. For all ${\mathcal E}_i$
we have the 1-densities $D_1({\mathcal E}_i)$, see Section \ref{preliminaries}. 

\begin{proposition} The avarage number of solutions of systems (\ref{syst}) is equal to
$2^{-n}\int _XD_1({\mathcal E}_1)\cdot \ldots \cdot D_1({\mathcal E}_n)$.
\end{proposition}
 
\begin{proof} Assume first that $\theta _i$ is an embedding. Then the assertion follows from
Crofton formula for the product of Euclidean spaces (Theorem \ref{gelfand-euclid}) and the property 
of normal densities with respect to inverse image of Finsler convex sets, see (\ref{***}).
If $\theta _i$ is not an embedding then the proof follows by application of Sard's lemma as
it is done for the first version of BKK theorem, see \cite{AK}, Sect. 4.4.
\end{proof}

\begin{theorem} The average number of solutions of systems (\ref{syst}) is equal to the mixed symplectic volume 
of Finsler ellipsoids ${\mathcal E}_1, \ldots, {\mathcal E}_n$ multiplied by $n!/ {2^n}$.
\end{theorem}
\begin{proof} This follows from the preceding result and equality (\ref{**}).
\end{proof}


\begin{thebibliography}{References}
\bibitem[1]{AK} D.\,Akhiezer, B.\,Kazarnovskii, {\it Average number of zeros and mixed symplectic volume of Finsler sets},
Geom. Funct. Anal., vol. 28 (2018), pp.1517--1547.

\bibitem[2]{Al} S.\,Alesker, {\it Theory of valuations on manifolds: a survey}, Geom. Funct. Anal., vol. 17 (2007), p. 1321--1341.  

\bibitem[3]{AB} S.\,Alesker, J.\,Bernstein, {\it Range characterization of the cosine tranform on higher Grassmanians},
Advances in Math. 184, no. 2 (2004), pp.367--379.  

\bibitem[4]{PF} J.-C.\,\' Alvarez Paiva, E.\,Fernandes,
{\it Gelfand transforms and Crofton formulas}, Selecta Math., vol. 13, no.3
(2008), pp.369 -- 390.

\bibitem[5]{Be} D.N.\,Bernstein, {\it The number of roots of a system of equations}, Funct. Anal. Appl. 9, no.2 (1975),
pp.95--96 (in Russian); Funct. Anal. Appl. 9, no.3 (1975), pp.183--185 (English translation).

\bibitem[6]{GS} I.M.\,Gelfand, M.M.\,Smirnov, {\it Lagrangians satisfying Crofton formulas, Radon transforms,
and nonlocal differentials}, Advances in Math. 109, no.2 (1994), pp. 188 -- 227.
  
\bibitem [7]{Sa} L.A.\,Santal\'o, {\it Integral Geometry and Geometric 
Probability}, Addison-Wesley, 1976.

\bibitem [8]{ZK} D.\,Zaporozhez, Z.\,Kabluchko, {\it Random determinants, mixed volumes of ellipsoids,
and zeros of Gaussian random fields}, Journal of Math. Sci., vol. 199, no.2 (2014), pp. 168--173. 


\end {thebibliography}
\end {document}